\documentclass[12pt,reqno]{amsart}
\usepackage{cite}
\usepackage{amssymb,latexsym,amsmath,amsfonts}
\usepackage{latexsym}
\usepackage{amsmath, amsthm, amscd, amsfonts, amssymb, graphicx, color}
\usepackage[mathscr]{eucal}
= 8.9in
\theoremstyle{plain}
\newtheorem{theorem}{\textbf{Theorem}}[section]

\newtheorem{problem}[theorem]{Open Problem}

\newtheorem{remark}[theorem]{\textbf{Remark}}

\newtheorem{definition}[theorem]{Definition}

\newtheorem{example}[theorem]{\textbf{Example}}

\numberwithin{equation}{section}
\begin{document}

\title[Weak orthogonal metric spaces and fixed point results ]
{Weak orthogonal metric spaces and fixed point results }
\author[T. Senapati]%
{Tanusri Senapati}

\address{{$^{1}$\,} Tanusri Senapati,
                    Department of Mathematics,
                    Indian Institute of Technology Guwahati,
                    Assam,
                    India.}
                    \email{senapati.tanusri@gmail.com}



 \begin{abstract}
  In this article we extend the notion of orthogonal metric space to weak orthogonal metric space. Then we establish fixed point results for a mapping satisfying a more general contraction condition.
  Several nontrivial examples are given in support of our obtained results. Moreover, we are able to answer of the open question  posed by Eshaghi et al. [On orthogonal sets and Banach fixed point theorem, Fixed Point Theory, 18(2017), 569-578].
 \end{abstract}
 \maketitle
 
 \section{Introduction and Preliminaries}
 This section is a prelude which leads us into the main results. Here we present some basic definitions and results that are 
 prerequisite  for the  main results of this manuscript. We begin this section by recalling the definition of an orthogonal set.
 \begin{definition}{\rm \cite{gor}}
  Let $X$ be a non empty set and $\perp$ be a binary relation defined on $X\times X$. Then $(X,\perp)$ is said to be an orthogonal 
  set (briefly, $O$-set) if there exists $x_0\in X$ such that 
  \[(\forall y\in X, x_0\perp y)~\mbox{or}~(\forall y\in X, y\perp x_0).\]
  The element $x_0$ is called an orthogonal element. An orthogonal set may have more than one orthogonal element.
 \end{definition}
\begin{example}
Let $X$ be a normed linear space.  We define $x\perp y$ if $||x+\lambda y||\geq ||x||$ for all $\lambda\in \mathbb{C}$. Then for all $y\in X,$ there exists $x=\theta\in X$ such that $||x+\lambda y||\geq ||x||$ for all $\lambda\in \mathbb{C}$. This shows that $(X,\perp)$ is an orthogonal set. 
 \end{example}
 For more examples and properties of orthogonal sets and orthogonal metric spaces, the reader are refereed to see {\rm\cite{gor,bagh}}. Now we introduce the definition of a weak orthogonal set.
 \begin{definition}
  Let $X$ be a non empty set and $\perp$ be a binary relation defined on $X\times X$. Then $(X,\perp)$ is said to be a weak orthogonal  set (briefly, $O_w$-set) if there exists $x_0\in X$ such that $\forall y\in X,$
 \[ x_0\perp y~\mbox{or}~ y\perp x_0.\]
 \end{definition}
The element $x_0$ is called a weak orthogonal element. Likewise an orthogonal set, a weak orthogonal set has more than one weak orthogonal element.\\

Two elements $x,y\in X$ are said to be orthogonally related if $x\perp y$ or $y\perp x$.
\begin{remark}
 From the definition, it is clear that every orthogonal set is a weak orthogonal set but the converse is not true. The following examples show the a weak orthogonal set is not an orthogonal set.
\end{remark}
\begin{example}\label{exam1}
 Let us set $X=\mathbb{R}$ and we define a binary relation $\perp$ on $X$ by
 \[x\perp y ~\mbox{if} ~x\leq y.\]
 It is very easy to check that $\perp$ is a weak orthogonal relation but it is not an orthogonal relation.
 For all $x\in X$ with $x\geq 0$, we have $0\perp x$ and for all $x\leq 0$, we have $x\leq 0$. Hence, $(\mathbb{R},\perp)$ is a weak orthogonal set. 
 Note that this set is not an orthogonal set since there exists no element $x_0\in X$ such that  for all $x\in X$, $x_0\perp x$ or for all $x\in X$,  $x\perp x_0$ holds. Also note that every element in $X$ is a weak orthogonal element.
\end{example}
\begin{example}
 Let us consider the linear space $M_{n\times n}(\mathbb{R})$ and $S=\{A\in M_{n\times n}(\mathbb{R}):A\geq 0~\mbox{or}~ A\leq 0\}.
$ Now we define a binary relation $\perp$ on $S$ as $A\perp B$ if $A-B\geq 0$. Clearly for all positive semidefinite matrices $A\in S$, $A\perp 0$ 
and for all negetive semidefinite matrices $A\in S$,  $0\perp A$. Therefore $(S,\perp)$ is a weak orthogonal set. 
\end{example}
\begin{example}
 Let $H$ be a an infinite dimensional Hilbert space and $S=\{P, I+P: P ~\mbox{is~ an ~orthogonal~ projection~ operator}\}$. Now we define a binary
 relation $\perp$ on $S$ as $P_1\perp P_2$ if $P_1\geq P_2$. Therefore for all $P\in S$, we have either $P\perp I$ or $I\perp P$. Hence $(S,\perp)$ is a weak orthogonal set.
\end{example}
In the following lines, we extend the notions of orthogonal sequence and Cauchy orthogonal sequence  to weak orthogonal sequence and Cauchy weak orthogonal sequence respectively.
\begin{definition}
Let $(X,\perp)$ be  a weak orthogonal set (briefly, $O_w$-set). A sequence $\{x_n\}_{n\in \mathbb{N}} \in X$ is said to be a weak orthogonal sequence 
(briefly, $O_w$-sequence) if $\forall n\in \mathbb{N},$
\[ x_n\perp x_{n+1}~\mbox{or}~x_{n+1}\perp x_{n}.\]
Similarly, a Cauchy sequence $\{x_n\}_{n\in \mathbb{N}}$ in $X$ is said to be a Cauchy weak orthogonal sequence (briefly, Cauchy $O_w$-sequence) if  $\forall n\in \mathbb{N},$
 \[ x_n\perp x_{n+1}~\mbox{or}~x_{n+1}\perp x_{n}.\]
\end{definition}
\begin{remark}
 Every orthogonal sequence is a weak orthogonal sequence but the converse is not true. 
 \end{remark}
\begin{example}
 Let us consider the weak orthogonal set in Example \ref{exam1}. We consider a sequence  $\{x_n\}_{n\in \mathbb{N}}\in X$ by $x_n=(-1)^n\frac{1}{n}$ 
 for all $n\in \mathbb{N}$. Clearly, for all $m\in \mathbb{N}$ with $n=2m+1, x_n\perp x_{n+1}$ and $n=2m, x_{n+1}\perp x_n$. This shows 
 that $\{x_n\}_{n\in \mathbb{N}}$ is a weak orthogonal sequence but not an orthogonal sequence.  
\end{example}
\begin{definition}
 A weak orthogonal metric space $(X,\perp, d)$ is said to be a complete weak orthogonal metric space (briefly, $O_w$-complete) if every Cauchy $O_w$-sequence converges in $X$.
\end{definition}

\begin{definition}
 A self map $T$ on a weak orthogonal metric space $(X,\perp,d)$ is said to be weak orthogonality preserving (briefly, $O_w$-preserving)
 if $x\perp y\Rightarrow Tx\perp Ty ~\mbox{or}~Ty\perp Tx$ for all $x,y\in X$.
\end{definition}
The authors of \cite{gor} defined $O$-continuity and Banach $\perp$-contraction as follows:
\begin{definition}
 Let $(X,\perp,d)$ be an orthogonal metric space. A function $T:X\to X$ is said to be orthogonally continuous ($O$-continuous)
at $x$ if for each $O$-sequence $\{x_n\}_{n\in \mathbb{N}}$ converging to $x$ implies that $T(x_n)\to Tx$ as $n\to \infty$.
 \end{definition}
\begin{definition}
 Let $(X,\perp,d)$ be an orthogonal metric space. A function $T:X\to X$ is said to be an orthogonal Banach contraction (briefly, Banach $\perp$-contraction)
if \[d(Tx,Ty)\leq k d(x,y)\] for all $x,y\in X$ with $x\perp y$.
\end{definition}
Here we would like to draw the reader's attention to a basic difference between the Banach contraction in metric spaces and orthogonal Banach contraction in 
orthogonal metric spaces. It is very well known that in metric space, every Banach contraction mapping is continuous mapping. But in orthogonal metric space, Banach $\perp$-contraction 
does not give the guarantee of orthogonal continuity of the mapping. In this regard, we present the following simple example.
\begin{example}
 We consider the orthogonal metric space $(X,\perp,d)$ where $X=\mathbb{R}$ and $$x\perp y~\mbox{if}~xy\in \mathbb{Q}.$$
 Therefore, for all $x\in X$, there exists $0\in \mathbb{R}$ such that $0\perp x$ and hence, $(X,\perp, d)$ is an orthogonal set. We define a 
 mapping $T:X\to X$ by \\
 \[ T(x) = \left\{ \begin{array}{ll}
0, & { x\in \mathbb{Q}^c};\\
 \frac{x}{3}, & {otherwise}.\end{array} \right. \]
 
 At first we show that $T$ is a Banach $\perp$-contraction. For all nonzero $x,y\in X$ with $x\perp y$ implies either $x,y\in \mathbb{Q}$ or $x,y\in \mathbb{Q}^c$ which implies that \[d(Tx,Ty)=\frac{x-y}{3}\leq \frac{1}{3}d(x,y)\] or
 \[d(Tx,Ty)=0\leq kd(x,y),\forall k\in [0,1).\]  Let $x= 0$ and $y\in \mathbb{R}$ be a nonzero number. Then it is easy to check that $d(Tx,Ty)\leq kd(x,y)$ for some $k\in [0,1)$. Therefore $T$ is a Banach $\perp$-contraction. Note that $T$ is not a Banach contraction. 
For example, let $x=1$ and $y=1+\frac{1}{\sqrt{11}}$. Then there exists no $k\in [0,1)$  such that \[d(Tx,Ty)=\frac{1}{3}\leq k d(x,y)\] holds.
Next, we claim that the mapping $T$ is not $O$-continuous. To show this, we consider the sequence $\{x_n\}_{n\in \mathbb{N}}$ in $X$ where $x_n=1+\frac{1}{1!}+
\frac{1}{2!}+\dots+\frac{1}{n!}$, for each $n\in \mathbb{N}$. Clearly, $\{x_n\}_{n\in \mathbb{N}}$ is an orthogonal sequence converging to $e$. It is easy to check that 
$T{x_n}\to \frac{e}{3} \neq T(e)=0$, i.e., $T$ is not $O$-continuous.
 \end{example}
Therefore to establish fixed point results in orthogonal metric space, we need to assume the condition of $O$-continuity of the mapping 
which is already defined in \cite{gor}. Now we are interested to extend the idea of $O$-continuity to orbitally $O$-continuiuty and then orbitally weak $O$-continuity.

By the notation $O_T(x)$, we define the orbit of $T$ at $x\in X$, i.e.,
\[O_T(x)=\{ T^nx:n=0,1,2,\dots\}.\]
\begin{definition}
 Let $(X,\perp,d)$ be an $O$-metric space and $T$ be a self mapping on $X$. Then $T$ is said to be orbitally $O$-continuous at $z$ if every $O$-sequence $\{y_n\}_{n\in \mathbb{N}}\in O_T(x)$, 
 for any $x\in X$, \[y_n\to z\implies Ty_n\to Tz.\]
\end{definition}
\begin{definition}
 Let $(X,\perp,d)$ be an $O$-metric space and $T$ be a self mapping on $X$. Then $X$ is said to be $T$-orbitally $O$-complete if every Cauchy $O$-sequence $\{y_n\}_{n\in \mathbb{N}}\in O_T(x)$, 
 for any $x\in X$, converges in $X$.
\end{definition}
\begin{example}
Let $X=(0,\infty)$  and we define $x\perp y$ if $xy\leq x~\mbox{or}~y$. Then for all $y\in X$, there exists $x=1$, such that $xy\leq y$. So, $(X,\perp)$ is an $O$-set. We consider the usual metric $d$ on $X$. Then $(X,\perp,d)$ is an $O$-metric space. Let $T:X\to X$ be defined as
 \[ T(x) = \left\{ \begin{array}{ll}
2, & { x\in (0,1)};\\
1, & { x=1};\\
 \frac{1}{3}, & {otherwise}.\end{array} \right. \]
Here, we claim that 
\begin{enumerate}
\item[(A)] The space $X$ is a $T$-orbitally $O$-complete metric space but not $O$-complete.
\item[(B)] The function $T$ is orbitally $O$-continuous but not $O$-continuous.
\end{enumerate} 
\begin{proof}
\noindent{\textbf{(A)}} To prove this, we consider the following cases:

\noindent{\textbf{Case-I:}} Let us consider $x\in (0,1)$. Then 
\begin{eqnarray*}
O_T(x)&=&\{T^nx:n=0,1,2,\dots\}\\
&=&\{x,2,\frac{1}{3},2,\frac{1}{3},\dots\}.
\end{eqnarray*}
Similarly for $x>1$, \[O_T(x)=\{x,,\frac{1}{3},2,\frac{1}{3},2,\dots\}.\]
Therefore for all $x\in (0,1)\vee (1,\infty), O_T(x)$ contains two subsequences. Subsequence $\{y_n\}_{n\in \mathbb{N}}=\{\frac{1}{3}\}, n\in \mathbb{N}$ is only Cauchy $O$-sequence which converges in $X$.

\noindent{\textbf{Case-II:}} For $x=1, O_T(x)=\{1,1,1,\dots\}$ contains a constant sequence which is Cauchy $O$-sequence.

The above two cases deduce that $(X,\perp,d)$ is a $T$-orbitally $O$-complete metric space. Let us consider a sequence $\{x_n\}_{n\in \mathbb{N}}$ in $X$ such that $x_n=\frac{1}{n}$ for all $n\in \mathbb{N}$. 
Clearly this sequence is Cauchy $O$-sequence but it is not convergent in $X$. Therefore, $(X,\perp,d)$ is not $O$-complete.

\noindent{\textbf{(B)}}
 We consider a sequence $\{x_n\}_{n\in \mathbb{N}}$ in $X$ such that $x_n=1-\frac{1}{n+1}$ for all $n\in \mathbb{N}$. Clearly this sequence is $O$-sequence and convergent to 1.
 For all $n\in \mathbb{N}$, $T{x_n}=2$ and $T1=1$ which implies that $T$ is not an $O$-continuous function. It is easy to check that $T$ is orbitally $O$-continuous function.
\end{proof}
\end{example}
Subsequently, we define the followings:
\begin{definition}
 Let $(X,\perp,d)$ be a $O_w$-metric space and $T$ be a self mapping on $X$. Then $T$ is said to be orbitally $O_w$-continuous at $z$ 
 if every $O_w$-sequence $\{x_n\}_{n\in \mathbb{N}}\in O_T(x)$, for any $x\in X$, \[y_n\to z\implies Ty_n\to Tz.\]
\end{definition}
\begin{definition}
 Let $(X,\perp,d)$ be a $O_w$-metric space and $T$ be a self mapping on $X$. Then $X$ is said to be $T$-orbitally $O_w$-complete if every Cauchy $O_w$-sequence $\{y_n\}_{n\in \mathbb{N}}\in O_T(x)$, 
 for any $x\in X$, converges in $X$.
\end{definition}

\section{main results}
This section comes up with the definition of generalized orthogonal contraction in weak orthogonal metric space and it presents a fixed point result concerning the maps.
\begin{definition}
Let $(X,\perp,d)$ be an $O_w$-metric space and $T$  be a self map on $X$. Then $T$ is said to be a generalized $\perp$-contraction if  \[d(Tx,Ty) \leq k M(x,y)\] for all orthogonally related elements $x,y\in X$ and
\begin{eqnarray*}
M(x,y)&=& \max\{d(x,y),d(x,Tx),d(y,Ty),\frac{d(x,Ty)+d(Tx,y)}{2},\\
&& \frac{d(T^2x,x)+d(T^2x,Ty)}{2},d(T^2x,Tx),d(T^2x,y),d(T^2x,Ty)\}.
\end{eqnarray*}
\end{definition}
\begin{theorem}\label{thm1}
Let $T$ be a self map on a weak orthogonal metric space $(X,\perp,d)$ and $X$ be a $T$-orbitally $O_w$-complete space. If $T$ is weak $\perp$-preserving, orbitally $O_w$-continuous and generalized $\perp$-contraction for some $k\in [0,1)$, then $T$ has a unique fixed point.
\end{theorem}
\begin{proof}
Since, $X$ is a weak orthogonal set, there exists at least one element $x_0\in X$ such that \[\forall y\in X, (x_0\perp y~\mbox{or}~y\perp x_0).\]
This implies that $x_0\perp Tx_0$ or $Tx_0\perp x_0$. Let us consider the iterated sequence $\{x_n\}_{n\in \mathbb{N}}$ where $x_n=T^nx_0$ for all $n\in \mathbb{N}$. Since $T$ is a weak $\perp$- preserving map, we must have either $T^nx_0\perp T^{n+1}x_0$ or   $T^{n+1}x_0\perp T^{n}x_0$ for all $n\in \mathbb{N}$, i.e.,$\{x_n\}_{n\in \mathbb{N}}$ is a weak $O$-sequence. Now we obtain
\begin{eqnarray*}
d(x_{n+1},x_{n+2}) &\leq & k M(x_{n},x_{n+1})\\
&\leq & k \max \big\{d(x_{n},x_{n+1}),d(x_{n},x_{n+1}),d(x_{n+1},x_{n+2}),\\
&& \frac{d(x_{n},x_{n+2})+d(x_{n+1},x_{n+1})}{2},\frac{d(x_{n+2},x_{n})+d(x_{n+2},x_{n+2})}{2},\\
&&d(x_{n+2},x_{n+1}),d(x_{n+2},x_{n+1}),d(x_{n+2},x_{n+2})\big\}\\
&\leq & \max\{d(x_{n},x_{n+1}),d(x_{n+1},x_{n+2})\}.
\end{eqnarray*}
Therefoe we must have \[d(x_{n+1},x_{n+2})\leq k d(x_{n},x_{n+1})\] for all $n\in \mathbb{N}$. Hence, one can obtain \[d(x_{n},x_{n+1})\leq k^nd(x_0,x_1)\] and $d(x_{n},x_{n+1})\to 0$ as $n\to \infty$. Next, we show that  $\{x_n\}_{n\in \mathbb{N}}$ is a Cauchy $O_w$-sequence. For all $m>n$,
\begin{eqnarray*}
d(x_n,x_m)&\leq & d(x_n,x_{n+1})+d(x_{n+1},x_{n+2})+\dots+d(x_{m-1},x_{m})\\
&\leq & \frac{k^n}{1-k}d(x_0,x_1) \to 0~\mbox{as}~n\to \infty.
\end{eqnarray*}
This shows that $\{x_n\}_{n\in \mathbb{N}}$ is a Cauchy $O_w$-sequence. Since, the space $X$ is $T$-orbitally $O_w$-complete, there exists some $z\in X$ such that $x_n\to z$ as $n\to \infty$. We claim that $z$ is a fixed point of $T$.

Given that $T$ is orbitally $O_w$-continuous function, i.e., for every weak $O$-sequence $\{x_n\}_{n\in \mathbb{N}}\in O_T(x_0)$ converging to $y$, we have $Ty_n\to Ty$ as $n\to \infty$. Since the sequence $\{x_n\}_{n\in \mathbb{N}}$  is itself a $O_w$-sequence converging to $z$, we must have that $Tx_n\to Tz$ as $n\to \infty$. Therefore, $Tz=\lim_{n\to \infty}Tx_n=\lim_{n\to \infty}x_{n+1}=z$, i.e., $z$ is a fixed point of $T$.

Finally, we prove the uniqueness of fixed points.  Let us consider $w$ is an another fixed point of $T$. Then we have either $x_0\perp w$ or $w\perp x_0$. As $T$ is a weak orthogonality preserving mapping, for all $n\in \mathbb{N},   x_n\perp w$ or $w\perp x_n.$ Then 
\begin{eqnarray*}
d(x_n,w)&=&d(Tx_{n-1},Tw)\\
&\leq & k M(x_{n-1},w)\\
& \leq & k \max \big\{d(x_{n-1},w),d(w,w),d(x_{n-1},x_n),\\
&& \frac{d(x_{n-1},w)+d(x_n,w)}{2},\frac{d(x_{n+1},x_{n-1})+d(x_{n+1},w)}{2},\\
&& d(x_{n+1},x_n), d(x_{n+1},w), d(x_{n+1},w)\big\}\\
&\leq & k\max \{d(x_{n-1},w), d(x_n,w),d(x_{n+1},w),d(x_{n-1},x_n),\\
&&d(x_{n-1},x_{n+1}),d(x_{n},x_{n+1})\}.
\end{eqnarray*}
Now we consider the following possibilities:
\begin{enumerate}
\item[(1)] Let $M(x_{n-1},w)=d(x_n,x_{n-1})$ or $d(x_n,x_{n+1}) $ or$d(x_{n-1},x_{n+1})$. Then one can immediately check that $d(x_n,w)\to  0$ as $n\to \infty$, i.e., the sequence $\{x_n\}_{n\in \mathbb{N}}$ converges to $w$.
\item[(2)]$M(x_{n-1},w)\neq d(x_{n},w)$  otherwise it leads to a contradiction as $k\in [0,1)$.
\item[(3)] Let $M(x_{n-1},w)=d(x_{n+1},w)$. Then we have 
\begin{eqnarray*}
d(x_n,w)&\leq & kd(x_{n+1},w)\leq k[d(x_{n+1},x_n)+d(x_n,w)]\\
d(x_n,w)&\leq & \frac{ k}{1-k}d(x_{n+1},x_n).
\end{eqnarray*}Passing through the limit $n\to \infty$ in the above inequality, we get $d(x_n,w)\to 0$, i.e., $\{x_n\}$ converges to $w$.
\item[(4)] Let $M(x_{n-1},w)=d(x_{n-1},w)$, so we obtain
\[d(x_n,w)\leq k d(x_{n-1},w).\]
Repeating the above process in a similar manner, we have
\begin{eqnarray*}
d(x_{n-1,w})&\leq &\frac{k}{1-k} d(x_{n},x_{n-1})\\
\Rightarrow d(x_n,w) &\leq &\frac{k^2}{1-k}d(x_{n},x_{n-1}) \to 0~\mbox{as}~n\to \infty
\end{eqnarray*}
or
\begin{eqnarray*}
d(x_{n-1},w) &\leq  & k d(x_{n-2},w)\\
\Rightarrow d(x_n,w) &\leq & k^2d(x_{n-2},w).
\end{eqnarray*}
Thus by routine calculation one cen observe that
\[d(x_n,w)\leq{k^n}d(x_{0},w) \] or \[d(x_n,w)\leq \frac{k^n}{1-k}d(x_{1},x_{0}).\] Passing through the limit $n\to \infty$ in the both cases, we obtain $x_n\to w$.
\end{enumerate}
Therefore we observe that the sequence $\{x_n\}_{n\in \mathbb{N}}$ converges to $w$. Since the limit of a sequence is unique, we must have that $z=w$.
\end{proof}
The existence of a fixed point of the mapping $T$ in the above theorem can be established under the following condition instead of orbitally $O_w$-continuity of $T$.\\

\noindent {\textbf{(O1)}} Suppose $\{x_n\}_{n\in \mathbb{N}}$ is a $O_w$-sequence in $O_T(x)$, for some $x\in X$, converging to $z$. Then $\{x_n\}_{n\in \mathbb{N}}$ has a subsequence $\{x_{n_k}\}_{k\in \mathbb{N}}$ such that  $\forall k\in \mathbb{N}$, \[x_{n_k}\perp z~\mbox{or}~ z \perp  x_{n_k}.\]

\begin{theorem}
Let $T$ be a self map on a weak orthogonal metric space $(x,\perp,d)$ and $X$ be $T$-orbitally $O_w$-complete. If $T$ is a weak orthogonally preserving, generalized $\perp$-contraction for some $k\in [0,1)$ and satisfies condition $(O1)$ then $T$ has a  unique fixed point.
\end{theorem}
\begin{proof}Continuing in a similar fashion of the proof of the above theorem, let us consider that the Cauchy $O_w$-sequence converges to $z$. We prove that $z$ is a fixed point of $T$. If possible, let $z$ be not a fixed point of $T$. Then we must have that $d(z,Tz)=\lambda>0$. By the property $(O1), x_{n_k}\perp z$ or $z\perp x_{n_k}$ for all $k\in \mathbb{N}$ which implies that $Tx_{n_k}\perp Tz$ or $Tz\perp Tx_{n_k}$ for all $k\in \mathbb{N}$. By using the generalized contraction of $T$, we deduce that
\begin{eqnarray*}
d(Tx_{n_k},Tz)& \leq & k M(x_{n_k}, z)\\
&\leq & k \max \big\{d(x_{n_k},z),d(x_{n_k},x_{n_k+1}),d(z,Tz),\\
&& \frac{d(x_{n_k},Tz)+d(x_{n_k+1},z)}{2}, \frac{d(x_{n_k+2},x_{n_k})+d(x_{n_k+2},Tz)}{2},\\
&& d(x_{n_k+2},x_{n_k+1}),d(x_{n_k+2},z),d(x_{n_k+2},Tz) \big\}.
\end{eqnarray*}
Passing through the limit $n\to \infty$ in the both sides of the inequality, we get 
\begin{eqnarray*}
\lim_{n\to \infty}d(Tx_{n_k}, Tz) &\leq & k d(z,Tz)=k\lambda< \lambda\\
\implies d(z,Tz)& \leq & \lambda, ~\mbox{a~contradiction.}
\end{eqnarray*}
Therefore, we must have that $z$ is a fixed point of $T$. The uniqueness of fixed point can be proved in a similar way of Theorem \ref{thm1}.
\end{proof}

\begin{remark}
It is worth to note that the contraction condition which we consider here  is more general than the contraction condition due to \'Ciri\'c \cite{ciric}. Therefore, one can easily access the fixed point result for the mapping satisfying \'Ciri\'c contraction condition from our results in weak orthogonal metric space. Also, we can obtain fixed point results for the mappings satisfying Kannan contraction \cite{kan} and Chatterjea contraction  \cite{chat} from the above results in this structure.  
\end{remark}

In support of our main result, we present the following example.
\begin{example}
Let us set $X=\{0,1,2,3,4\}$ and consider an arbitrary binary relation $\mathcal{R}$ on $X$ as 
\[\mathcal{R}=\{(0,0),(1,0),(0,2),(3,4),(3,0),(4,0)\}.\]
For any two elements $x,y\in X, x\perp y$ if $(x,y)\in \mathcal{R}$. Clearly, $(X,\perp)$ is not an orthogonal set but it is a weak orthogonal set as for all $x\in X$, there exists $y=0$ such that $(0,x)\in \mathcal{R}$ or  $(x,0)\in \mathcal{R}.$\\
We define a mapping $T:X\to X$ by \[T0=0,T1=0, T2=1,T3=0,T4=2.\]
Except $(x,y)=(4,0)$, for all $(x,y)\in \mathcal{R},(Tx,Ty)\in \mathcal{R}.$ Observe that $(T4,T0)=(2,0) \notin \mathcal{R}$ but $(0,2)\in \mathcal{R}$. This shows that $T$ is a weak $\perp$-preseving mapping.\\
Next we check the contraction condition. Note that for all $x,y$ with $x\perp y$, we have $d(Tx,Ty)\leq k d(x,y)$ for some $k\in [0,1)$ except the point $(3,4)$. We show that $T$ satisfies the generalized contraction condition. For $(x,y)=(3,4)$,
\begin{eqnarray*}
M(3,4) &=& \max\{d(3,4), d(3,T3),d(4,T4),\frac{d(3,T4)+d(T3,4)}{2},\\
&&\frac{d(T^23,3)+d(T^23,T4)}{2},d(T^23,T3),d(T^23,4),d(T^23,T4)\}\\
&=& \max\{d(3,4),d(3,0),d(2,4), \frac{d(3,2)+d(0,4)}{2}, \frac{d(0,3)+d(0,2)}{2},d(0,4),d(0,2)\}\\
&=&4.
\end{eqnarray*} Therefore, for all $x,y\in X$ with $x\perp y$, we have \[d(Tx,Ty)\leq kM(x,y),\] i.e., $T$ is a generalized $\perp$-contraction mapping.
Now,
$$O_T(0)=\{0,0,0,\dots\},$$
$$O_T(1)=\{1,0,0,\dots\},$$
$$O_T(2)=\{2,1,0,0,\dots\},$$
$$O_T(3)=\{3,0,0,\dots\},$$
$$O_T(4)=\{4,2,1,0,0,\dots\}.$$
Observe that for all $x\in X, O_T(x)$ contains a constant sequence. This implies that $X$ is $T$-orbitally $O_w$-complete metric space and $T$ is also $O_w$-continuous map. All the conditions of our theorem are satisfied. Here, $x=0$ is the unique fixed point of $T$.
\end{example}
\section{Answer to the question posed in \cite{gor}}
We have already mentioned that the authors of \cite{gor} defined the concept of $O$-continuity. They proved that every continuous function is $O$-continuous but the reverse implication does not hold in general. In that connection, they raised the following question on the inner product spaces.
\begin{problem}
Let $X$ be an inner product space with the inner product $\langle.,.\rangle$. We define an orthogonal relation $\perp$  on $X$ as $x\perp y$ if $\langle x,y\rangle=0.$ Let $f:X\to X$ be a $O$-continuous function. Is $f$ continuous?
\end{problem}
The authors of \cite{sena} tried to answer of this question and claimed that in an inner product space, every $O$-continuous function is continuous. Here we  reinvestigate that problem and observe that their claim was not right, i.e.,  
there may exists an $O$-continuous function which is not necessarily continuous  in  inner product spaces.. In this purpose, we construct the following example of an $O$-continuous function  in the standard inner product space $\mathbb{R}^2$ which is not continuous.
\begin{example}
Let $\big(X,\langle.,.\rangle\big)$ be a standard inner product space, where $X=\mathbb{R}^2$ and $\langle x,y\rangle=x_1y_1+x_2y_2$, for all $x=(x_1,x_2),y=(y_1,y_2)\in X$. An orthogonal relation on $X$ is defined as
\[x\perp y~\mbox{if}~\langle x,y\rangle=0.\] 
Clearly, $(X,\perp)$ is an orthogonal set as for all $x\in X, \langle \theta,x\rangle=0$ where $\theta=(0,0)$.\\
Let us define a function $F:X \to X$ by 
\[ F(x_1,x_2) = \left\{ \begin{array}{ll}
\big(\frac{x_1x_2}{x_1^2+x_2^2},0), & { (x_1,x_2)=(\frac{1}{n},\frac{1}{n+1}),n\in \mathbb{N}};\\
 (0,0), & {otherwise}.\end{array} \right. \]

We prove that this function is $O$-continuous  at $\theta=(0,0)$ but not continuous at that point. Before showing that, we claim the followings:
\begin{enumerate}
\item[(A)] there exists no orthogonal sequence $\{x^{(n)}_{n\in \mathbb{N}}\}$ such that $x^(n)=(\frac{1}{i+n},\frac{1}{i+n+1})$ for some $i\in \mathbb{N}_0$ and all $n\in \mathbb{N}$,
\item[(B)] for any $k\in \mathbb{N}, (\frac{1}{k},\frac{1}{k+1})$ can not be a limit point of any orthogonal sequence.
\end{enumerate}
 \begin{proof}\noindent{\textbf{(A)}} Let us consider there exists an $O$-sequence $\{x^{(n)}\}_{n\in \mathbb{N}}$   such that $x^{(n)}=(\frac{1}{i+n},\frac{1}{i+n+1})$ for some $i\in \mathbb{N}_0$ and all $n\in \mathbb{N}$. Then $\langle x^{(n)},x^{(n+1)}\rangle=\frac{1}{i+n}\frac{1}{i+n+1}+\frac{1}{i+n+1}\frac{1}{i+n+2}\neq 0$ for all $i_0\in \mathbb{N}_0$ and $n\in \mathbb{N}$ which contradicts that $\{x^{(n)}\}_{n\in \mathbb{N}}$ is an orthogonal sequence.\\
\noindent{\textbf{(B)}} If possible let $(\frac{1}{k},\frac{1}{k+1})$, for some $k\in \mathbb{N}$, be a limit point of an orthogonal sequence $\{x^{(n)}\}_{n\in \mathbb{N}}$. Let us choose a number $\epsilon>0$ such that $\epsilon<\frac{1}{k+1}$. Then for every such choice of $\epsilon>0$, we can found $n_0\in \mathbb{N}$ such that $x_1^{(n)}\in (\frac{1}{k}-\epsilon,\frac{1}{k}+\epsilon)$ and $x_2^{(n)}\in (\frac{1}{k+1}-\epsilon,\frac{1}{k+1}+\epsilon)$, i.e., $x_i^{(n)}>0$ for $i=1,2$ and all $n\geq n_0$. This implies that for all $n\geq n_0, \langle x^{(n)},x^{(n+1)}\rangle \neq 0$
which contradicts the orthogonality of the sequence $\{x^{(n)}\}_{n\in \mathbb{N}}$. Hence our assumption was wrong.
\end{proof}
Therefore any $O$-sequence $\{x^{(n)}\}_{n\in \mathbb{N}}$ converging to $z=(x,y)$, we must have that \[F(x_1^{(n)},x_2^{(n)})=(0,0)=F(x,y).\] for all $n\in \mathbb{N}$. This shows that $F$ is $O$-continuous at $z=(x,y)$. This implies $F$ is also $O$-continuous at $\theta=(0,0)$.\\
Next, we consider a sequence $\{y^{(n)}\}_{n\in \mathbb{N}}$ where $y^{(n)}=(\frac{1}{n},\frac{1}{n+1}),\forall n\in \mathbb{N}$. It is clear from $(A)$ that this sequence is not an $O$-sequence. Also the sequence  $\{y^{(n)}\}_{n\in \mathbb{N}}$  converges to $\theta=(0,0)$ as $n\to \infty$ but \[\lim_{n\to \infty}F(y_1^{(n)},y_2^{(n)})=(\frac{1}{2},0)\neq (0,0),\] i.e., $F$ is not continuous at $\theta=(0,0)$. 
\end{example}
In the standard inner product space $\mathbb{R}^n$, one can consider the following $O$-continuous function which is not continuous.
It is easy to check that the function $F:\mathbb{R}^n \to \mathbb{R}^n$ defined by 
\[ F(x_1,\dots,x_n) = \left\{ \begin{array}{ll}
\big(\frac{x_1x_2}{x_1^2+x_2^2},0,\dots,0\big), & { (x_1,x_2,x_3,\dots,x_n)=(\frac{1}{n},\frac{1}{n+1},0,\dots,0),n\in \mathbb{N}};\\
 (0,\dots,0), & {otherwise}.\end{array} \right. \]
is $O$-continuous at $\theta=(0,0, \dots,0)\in \mathbb{R}^n$ but it is not a continuous function.\\
Therefore, in general we can conclude that in arbitrary inner product spaces an $O$-continuous function may not be a continuous function.\\

\noindent{\textbf{Acknowledgement:}} This research work was funded by Institute Post Doctoral Programme (Project No.-IITG/R\&D/IPDF/2017-2018/MA01) of Indian Institute of Technology Guwahati, India. 
The author conveys her sincere thanks to Professor Arup Chattopadhyay for his valuable discussions and comments on this articlle.

\bibliography{bibfile}
\bibliographystyle{plain}

\end{document}